\documentclass[11pt]{amsart}
\usepackage{amsmath,amsfonts,mathrsfs}

\newif\ifdraft
\draftfalse
\usepackage{amsmath,amsfonts,amsthm,mathrsfs}
\usepackage{amssymb}

\usepackage[unicode,bookmarks]{hyperref}
\usepackage[usenames,dvipsnames]{xcolor}
\hypersetup{colorlinks=true,citecolor=NavyBlue,linkcolor=Brown,urlcolor=Orange}

\usepackage[alphabetic,initials]{amsrefs}

\usepackage{enumitem}

\usepackage{chngcntr}

\ifdraft
\usepackage[notcite,notref,color]{showkeys}

\definecolor{labelkey}{gray}{0.5}
\fi

\usepackage{tikz}

\usepackage{tikz-cd}

\usetikzlibrary{matrix,arrows}
\newlength{\myarrowsize} 

\pgfarrowsdeclare{cmto}{cmto}{
	\pgfsetdash{}{0pt} 
	\pgfsetbeveljoin 
	\pgfsetroundcap 
	\setlength{\myarrowsize}{0.6pt}
	\addtolength{\myarrowsize}{.5\pgflinewidth}
	\pgfarrowsleftextend{-4\myarrowsize-.5\pgflinewidth} 
	\pgfarrowsrightextend{.8\pgflinewidth}
}{
	\setlength{\myarrowsize}{0.6pt} 
  	\addtolength{\myarrowsize}{.5\pgflinewidth}  
	\pgfsetlinewidth{0.5\pgflinewidth}
	\pgfsetroundjoin
	\pgfpathmoveto{\pgfpoint{1.5\pgflinewidth}{0}}
	\pgfpatharc{-109}{-170}{4\myarrowsize}
	\pgfpatharc{10}{189}{0.58\pgflinewidth and 0.2\pgflinewidth}
	\pgfpatharc{-170}{-115}{4\myarrowsize+\pgflinewidth}
	\pgfpathclose
	\pgfusepathqfillstroke
	\pgfpathmoveto{\pgfpoint{1.5\pgflinewidth}{0}}
	\pgfpatharc{109}{170}{4\myarrowsize}
	\pgfpatharc{-10}{-189}{0.58\pgflinewidth and 0.2\pgflinewidth}
	\pgfpatharc{170}{115}{4\myarrowsize+\pgflinewidth}
	\pgfpathclose
	\pgfusepathqfillstroke
	\pgfsetlinewidth{2\pgflinewidth}
}

\pgfarrowsdeclare{cmonto}{cmonto}{
	\pgfsetdash{}{0pt} 
	\pgfsetbeveljoin 
	\pgfsetroundcap 
	\setlength{\myarrowsize}{0.6pt}
	\addtolength{\myarrowsize}{.5\pgflinewidth}
	\pgfarrowsleftextend{-4\myarrowsize-.5\pgflinewidth} 
	\pgfarrowsrightextend{.8\pgflinewidth}
}{
	\setlength{\myarrowsize}{0.6pt} 
  	\addtolength{\myarrowsize}{.5\pgflinewidth}  
	\pgfsetlinewidth{0.5\pgflinewidth}
	\pgfsetroundjoin
	\pgfpathmoveto{\pgfpoint{1.5\pgflinewidth}{0}}
	\pgfpatharc{-109}{-170}{4\myarrowsize}
	\pgfpatharc{10}{189}{0.58\pgflinewidth and 0.2\pgflinewidth}
	\pgfpatharc{-170}{-115}{4\myarrowsize+\pgflinewidth}
	\pgfpathclose
	\pgfusepathqfillstroke
	\pgfpathmoveto{\pgfpoint{1.5\pgflinewidth}{0}}
	\pgfpatharc{109}{170}{4\myarrowsize}
	\pgfpatharc{-10}{-189}{0.58\pgflinewidth and 0.2\pgflinewidth}
	\pgfpatharc{170}{115}{4\myarrowsize+\pgflinewidth}
	\pgfpathclose
	\pgfusepathqfillstroke
	\pgfpathmoveto{\pgfpoint{1.5\pgflinewidth-0.3em}{0}}
	\pgfpatharc{-109}{-170}{4\myarrowsize}
	\pgfpatharc{10}{189}{0.58\pgflinewidth and 0.2\pgflinewidth}
	\pgfpatharc{-170}{-115}{4\myarrowsize+\pgflinewidth}
	\pgfpathclose
	\pgfusepathqfillstroke
	\pgfpathmoveto{\pgfpoint{1.5\pgflinewidth-0.3em}{0}}
	\pgfpatharc{109}{170}{4\myarrowsize}
	\pgfpatharc{-10}{-189}{0.58\pgflinewidth and 0.2\pgflinewidth}
	\pgfpatharc{170}{115}{4\myarrowsize+\pgflinewidth}
	\pgfpathclose
	\pgfusepathqfillstroke
	\pgfsetlinewidth{2\pgflinewidth}
}

\pgfarrowsdeclare{cmhook}{cmhook}{
	\pgfsetdash{}{0pt} 
	\pgfsetbeveljoin 
	\pgfsetroundcap 
	\setlength{\myarrowsize}{0.6pt}
	\addtolength{\myarrowsize}{.5\pgflinewidth}
	\pgfarrowsleftextend{-4\myarrowsize-.5\pgflinewidth} 
	\pgfarrowsrightextend{.8\pgflinewidth}
}{
	\setlength{\myarrowsize}{0.6pt} 
  	\addtolength{\myarrowsize}{.5\pgflinewidth}  
 	\pgfsetdash{}{0pt}
	\pgfsetroundcap
	\pgfpathmoveto{\pgfqpoint{0pt}{-4.667\pgflinewidth}}
	\pgfpathcurveto
    {\pgfqpoint{4\pgflinewidth}{-4.667\pgflinewidth}}
    {\pgfqpoint{4\pgflinewidth}{0pt}}
    {\pgfpointorigin}
	\pgfusepathqstroke
}

\newenvironment{diagram*}[2]{%
\[%
\begin{tikzpicture}[>=cmto,baseline=(current bounding box.center),%
	to/.style={->,font=\scriptsize,cap=round},%
	into/.style={cmhook->,font=\scriptsize,cap=round},%
	onto/.style={-cmonto,font=\scriptsize,cap=round},%
	math/.style={matrix of math nodes, row sep=#2, column sep=#1,%
		text height=1.5ex, text depth=0.25ex}]%
}{%
\end{tikzpicture}%
\]%
\ignorespacesafterend%
}

\newcommand{\Dmod}{\mathscr{D}}

\newcommand{\Mmod}{\mathcal{M}}

\newcommand{\shT}{\mathscr{T}}

\newcommand{\derR}{\mathbf{R}}
\newcommand{\derL}{\mathbf{L}}

\newcommand{\ZZ}{\mathbb{Z}}
\newcommand{\QQ}{\mathbb{Q}}

\newcommand{\CC}{\mathbb{C}}

\newcommand*{\sheafhom}{\mathscr{H}\kern -.5pt om}

\newcommand{\shf}[1]{\mathscr{#1}}

\def\overbar#1#2#3{{%
	\setbox0=\hbox{\displaystyle{#1}}%
	\dimen0=\wd0
	\advance\dimen0 by -#2 
	\vbox {\nointerlineskip \moveright #3 \vbox{\hrule height 0.3pt width \dimen0}%
		\nointerlineskip \vskip 1.5pt \box0}%
}}

\newcommand{\shO}{\shf{O}}

\makeatletter
\let\@@seccntformat\@seccntformat
\renewcommand*{\@seccntformat}[1]{%
  \expandafter\ifx\csname @seccntformat@#1\endcsname\relax
    \expandafter\@@seccntformat
  \else
    \expandafter
      \csname @seccntformat@#1\expandafter\endcsname
  \fi
    {#1}%
}
\newcommand*{\@seccntformat@subsection}[1]{%
  \textbf{\csname the#1\endcsname.}
}
\makeatother

\makeatletter
\let\@paragraph\paragraph
\renewcommand*{\paragraph}[1]{%
	\vspace{0.3\baselineskip}%
	\@paragraph{\textit{#1}}%
}
\makeatother

\counterwithin{equation}{subsection}
\counterwithout{subsection}{section}
\counterwithin{figure}{subsection}

\newtheorem{theorem}[equation]{Theorem}
\newtheorem*{theorem*}{Theorem}
\newtheorem{lemma}[equation]{Lemma}
\newtheorem*{lemma*}{Lemma}
\newtheorem{corollary}[equation]{Corollary}
\newtheorem{proposition}[equation]{Proposition}
\newtheorem*{proposition*}{Proposition}

\theoremstyle{definition}

\newtheorem*{definition*}{Definition}
\newtheorem{remark}[equation]{Remark}

\newtheorem*{example*}{Example}
\newtheorem*{problem*}{Problem}

\theoremstyle{plain}

\newcommand{\theoremref}[1]{\hyperref[#1]{Theorem~\ref*{#1}}}
\newcommand{\lemmaref}[1]{\hyperref[#1]{Lemma~\ref*{#1}}}
\newcommand{\definitionref}[1]{\hyperref[#1]{Definition~\ref*{#1}}}
\newcommand{\propositionref}[1]{\hyperref[#1]{Proposition~\ref*{#1}}}
\newcommand{\conjectureref}[1]{\hyperref[#1]{Conjecture~\ref*{#1}}}
\newcommand{\corollaryref}[1]{\hyperref[#1]{Corollary~\ref*{#1}}}
\newcommand{\exampleref}[1]{\hyperref[#1]{Example~\ref*{#1}}}

\makeatletter
\let\old@caption\caption
\renewcommand*{\caption}[1]{%
	\setcounter{figure}{\value{equation}}%
	\stepcounter{equation}%
	\old@caption{#1}\relax%
}
\makeatother

\newcounter{intro}

\newtheorem{intro-conjecture}[intro]{Conjecture}
\newtheorem{intro-corollary}[intro]{Corollary}
\newtheorem{intro-theorem}[intro]{Theorem}

\newcommand{\parref}[1]{\hyperref[#1]{\S\ref*{#1}}}

\makeatletter
\newcommand*\if@single[3]{%
  \setbox0\hbox{${\mathaccent"0362{#1}}^H$}%
  \setbox2\hbox{${\mathaccent"0362{\kern0pt#1}}^H$}%
  \ifdim\ht0=\ht2 #3\else #2\fi
  }
\newcommand*\rel@kern[1]{\kern#1\dimexpr\macc@kerna}
\newcommand*\widebar[1]{\@ifnextchar^{{\wide@bar{#1}{0}}}{\wide@bar{#1}{1}}}
\newcommand*\wide@bar[2]{\if@single{#1}{\wide@bar@{#1}{#2}{1}}{\wide@bar@{#1}{#2}{2}}}
\newcommand*\wide@bar@[3]{%
  \begingroup
  \def\mathaccent##1##2{%
    \if#32 \let\macc@nucleus\first@char \fi
    \setbox\z@\hbox{$\macc@style{\macc@nucleus}_{}$}%
    \setbox\tw@\hbox{$\macc@style{\macc@nucleus}{}_{}$}%
    \dimen@\wd\tw@
    \advance\dimen@-\wd\z@
    \divide\dimen@ 3
    \@tempdima\wd\tw@
    \advance\@tempdima-\scriptspace
    \divide\@tempdima 10
    \advance\dimen@-\@tempdima
    \ifdim\dimen@>\z@ \dimen@0pt\fi
    \rel@kern{0.6}\kern-\dimen@
    \if#31
      \overline{\rel@kern{-0.6}\kern\dimen@\macc@nucleus\rel@kern{0.4}\kern\dimen@}%
      \advance\dimen@0.4\dimexpr\macc@kerna
      \let\final@kern#2%
      \ifdim\dimen@<\z@ \let\final@kern1\fi
      \if\final@kern1 \kern-\dimen@\fi
    \else
      \overline{\rel@kern{-0.6}\kern\dimen@#1}%
    \fi
  }%
  \macc@depth\@ne
  \let\math@bgroup\@empty \let\math@egroup\macc@set@skewchar
  \mathsurround\z@ \frozen@everymath{\mathgroup\macc@group\relax}%
  \macc@set@skewchar\relax
  \let\mathaccentV\macc@nested@a
  \if#31
    \macc@nested@a\relax111{#1}%
  \else
    \def\gobble@till@marker##1\endmarker{}%
    \futurelet\first@char\gobble@till@marker#1\endmarker
    \ifcat\noexpand\first@char A\else
      \def\first@char{}%
    \fi
    \macc@nested@a\relax111{\first@char}%
  \fi
  \endgroup
}
\makeatother

\newcommand{\I}{\mathcal{I}}

\def\ZZ{{\mathbf Z}}

\def\CC{{\mathbf C}}

\def\QQ{{\mathbf Q}}

\newtheorem*{thmA'}{Theorem~A$^\prime$}

\begin{document}

\vspace{\baselineskip}

\title{Hodge filtration, minimal exponent, and local vanishing}

\author[M. Musta\c{t}\u{a}]{Mircea~Musta\c{t}\u{a}}
\address{Department of Mathematics, University of Michigan, 530 Church Street,
Ann Arbor, MI 48109, USA}
\email{{\tt mmustata@umich.edu}}

\author[M.~Popa]{Mihnea~Popa}
\address{Department of Mathematics, Northwestern University, 
2033 Sheridan Road, Evanston, IL
60208, USA} \email{{\tt mpopa@math.northwestern.edu}}

\thanks{MM was partially supported by NSF grant DMS-1701622 and a Simons Fellowship; MP was partially supported by NSF grant
DMS-1700819.}


\begin{abstract}
We bound the generation level of the Hodge filtration on the localization along a hypersurface in terms of its minimal exponent. As a consequence, we obtain a local vanishing theorem for sheaves of forms with log poles. 
These results are extended to $\QQ$-divisors, and are derived from a result of independent interest on the generation level of the Hodge filtration on nearby and vanishing cycles.
\end{abstract}

\maketitle

\makeatletter

\section{Introduction}
Let $X$ be a smooth complex variety of dimension $n$, and $\Dmod_X$ the sheaf of differential
operators on $X$. An important invariant of a filtered $\Dmod_X$-module
$(\Mmod,F)$ of geometric origin is the complexity of its filtration, namely 
how many steps are required to fully determine it. Concretely, the filtration $F$ is generated at level $q$ if
$$F_{\ell}\Dmod_X\cdot F_q\Mmod =F_{q+\ell}\Mmod \quad\text{for all}\quad \ell\geq 0.$$
Here $F_\bullet \Dmod_X$ denotes the standard filtration by the order of differential operators. 

In this paper we give a bound for the generation level of the Hodge filtration on $\Dmod_X$-modules naturally associated to rational multiples of a reduced effective divisor $D$ on $X$, in terms of data provided by the Bernstein-Sato polynomial of $D$. 
This study was initiated by Saito \cite{Saito-HF}, who provided such bounds for special types of singularities. Some general results were later found in \cite{MP1}, \cite{MP2}. We improve them here, using the main result of \cite{MP3}, and also exploit the fact that they are, somewhat surprisingly, related to local vanishing theorems for sheaves of forms with log poles in birational geometry. 

\noindent
{\bf Reduced divisors.}
To highlight the main points with a minimum amount of technicalities, we first restrict our discussion to 
the case when we simply deal with a reduced effective divisor $D$. The corresponding $\Dmod_X$-module is the localization $\shO_X(*D)$, that is, the sheaf of functions with poles of arbitrary order along $D$. 
It is well known that $\shO_X(*D)$ is regular holonomic, and underlies a mixed Hodge module on $X$; therefore it comes endowed with a \emph{Hodge filtration} $F_p \shO_X(*D)$, with $p \ge 0$. See e.g. \cite{MP1} for an in-depth study of this filtration. 
If $D$ is smooth, then the filtration is generated at level $0$, hence from now on we focus on the case when $D$ is singular. We prove:

\begin{intro-theorem}\label{main1}
For every singular divisor $D$, the Hodge filtration on $\shO_X(*D)$ is generated at level $n - 1 - \lceil \widetilde{\alpha}_D \rceil$. 
\end{intro-theorem}

Here $\widetilde{\alpha}_D$ is the \emph{minimal exponent} of $D$, a positive rational number which is defined as the negative of the largest root of the \emph{reduced} Bernstein-Sato polynomial $\tilde{b}_D (s)$; see e.g. \cite{Saito-B}. It is a refined version of the log canonical threshold of the pair $(X, D)$, which is equal to ${\rm min}\{\widetilde\alpha_D, 1\}$. See \S\ref{scn:general} for further details and references. It was Saito who 
 first pointed out in \cite{Saito-HF} the relevance of the invariant $n - 1 - \lceil \widetilde{\alpha}_D \rceil$, proving the bound in Theorem \ref{main1} for isolated semi-quasihomogeneous singularities (when 
 $\widetilde{\alpha}_D$ can be computed explicitly).
 
Since $\widetilde{\alpha}_D > 0$, Theorem \ref{main1} recovers in particular the fact that $F_\bullet \shO_X(*D)$ is always generated at level $n-2$, proved in \cite[Theorem~B]{MP1}. Note also that it is possible to do better than Theorem \ref{main1}: as an extreme case, if $D$ is a singular simple normal crossing 
divisor, then  $F_\bullet \shO_X(*D)$ is generated at level $0$, but $\widetilde{\alpha}_D = 1$. The bound is nevertheless sometimes optimal; for instance, this is the case when $D$ has an isolated quasihomogeneous singularity by \cite[Theorem~0.7]{Saito-HF}.

Moreover, Saito \cite[Theorem~0.4]{Saito-B} showed that $\widetilde{\alpha}_D > 1$ is equivalent to $D$ having rational singularities, and therefore:

\begin{intro-corollary}
If $n\ge 3$ and the divisor $D$ has rational singularities, then the Hodge filtration on $\shO_X(*D)$ is generated at level $n-3$.\footnote{As mentioned above, for $n=2$ the filtration is always generated at level $0$.}
\end{intro-corollary}

This was proved when $D$ has isolated singularities, and conjectured to be true in general, in \cite{MOP}. The 
general conjecture was already verified recently by Kebekus-Schnell \cite[\S1.3]{KS}, as a consequence of a local vanishing conjecture; more on this below. Note that $\widetilde{\alpha}_D$ could however be much larger than $1$, and is in fact optimally bounded above by $n/2$ in \cite{Saito_microlocal} (see also \cite[Theorem~E]{MP3}).


It turns out that the generation level of the Hodge filtration on $\shO_X(*D)$ is intimately linked to a result in birational geometry, namely to local vanishing for pushforwards of bundles of forms with log poles. Consider a log resolution $\mu\colon Y \to X$ of the pair $(X,D)$, which is an isomorphism over $U = X \smallsetminus D$, and denote $E = (\mu^*D)_{\rm red}$.
We showed in \cite[Theorem~17.1]{MP1} that $F_\bullet \shO_X(*D)$ is generated at level $q$ if and only if
$R^i \mu_* \Omega^{n-i}_Y (\log E) = 0$ for $i > q$, so consequently we obtain:
 
\begin{intro-corollary}\label{main2}
With the above notation, we have
$$R^i \mu_* \Omega_Y^{n-i}(\log E) =0\quad\text{for}\quad i > n- 1 - \lceil\widetilde{\alpha}_D \rceil.$$
\end{intro-corollary}

When $i \ge n-1$ this is shown by elementary methods in \cite[Theorem~B]{MP1}, leading to the coarse bound $n-2$ for the generation level of the Hodge filtration mentioned above. When $D$ has rational singularities and $i = n-2$, it is proved in \cite{MOP} in the isolated singularities case, and can be deduced in general from a
vanishing statement obtained by Kebekus-Schnell \cite[Theorem~1.9]{KS}, which answers \cite[Conjecture~A]{MOP}. 
Using Corollary \ref{main2}, we can in fact obtain a strengthening of this conjecture/statement in the absolute case of a 
\emph{reduced singular hypersurface}: by this here we mean a singular complex scheme $D$, reduced but not necessarily irreducible,
 that can be embedded as a hypersurface in a smooth variety. In this case $D$ has an associated minimal exponent $\widetilde{\alpha}_D$, independent of the embedding (since this is the case already for the Bernstein-Sato polynomial). We consider a resolution of singularities $\mu\colon \widetilde{D} \to D$, given by the disjoint union of resolutions of the irreducible components of $D$. We further assume that $\mu$ is an isomorphism over the smooth locus of $D$ and  the reduced inverse image 
of the singular locus of $D$ is a simple normal crossing divisor $E$ on $\widetilde{D}$. We then have~\footnote{Note that $D$ has rational singularities if and only if $\widetilde{\alpha}_D > 1$, so the case $i = n-2$ corresponds to the statements in \emph{loc. cit.}}

\begin{intro-theorem}\label{local_strong}
With the above notation, if $\dim(D)=n-1$, then 
$$R^i  \mu_* \Omega^{n-1-i}_{\widetilde D} (\log E) = 0 \,\,\,\,\,\,{\rm for~all} \,\,\,\,\,\,i > n -1 - \lceil \widetilde{\alpha}_D \rceil.$$
\end{intro-theorem}

We emphasize that here the overall strategy is reversed: we first show the generation bound in Theorem \ref{main1} using methods from the theory of (Hodge) $\Dmod$-modules, and then deduce the birational Corollary \ref{main2}, which in turn is used to prove
Theorem~\ref{local_strong}. At the moment we do not know how to approach the latter vanishing results via more standard methods 
in birational geometry.

\noindent
{\bf Rational multiples.} Following  \cite{MP2}, \cite{MP3}, we also consider a multiple $\alpha D$, where 
$\alpha$ is a positive rational number and $D$ is a reduced effective divisor on $X$, as above. 
The set-up is local: assuming that $D$ is defined by a regular function $f$, the natural replacement for the localization $\shO_X [1/f]$ is the $\Dmod_X$-module
$$\Mmod ( f^{-\alpha}) : = \shO_X [1/f]  f^{-\alpha},$$
the free rank $1$ module over $\shO_X [1/f] $ generated by the formal symbol $f^{-\alpha}$; see \S\ref{scn:general}. This is a direct summand of a mixed Hodge module, and so analogously it comes endowed with a Hodge filtration $F_p \Mmod ( f^{-\alpha})$, with $p \ge 0$. Again, if $D$ is smooth, then this filtration
is generated at level $0$, hence from now on we focus on the case when $f$ defines a singular hypersurface. 

Theorem \ref{main1} and Corollary \ref{main2} above are then special cases (when $\alpha =1$) of the following two statements that will be the focus of the paper. 

\begin{intro-theorem}\label{main3}
If $f$ defines a singular reduced hypersurface, then 
the Hodge filtration on $\Mmod ( f^{-\alpha})$ is generated at level $n - \lceil \widetilde{\alpha}_f + 
\alpha \rceil$. 
\end{intro-theorem}

In the special case when $D$ has an isolated quasihomogeneous singularity, by analogy with the reduced case in \cite{Saito-HF}, this result was conjectured in \cite{Popa} and proved in \cite{Zhang}. Note also that Theorem~\ref{main3} recovers the second statement of \cite[Theorem~10.1]{MP2}, namely that the filtration on $\Mmod ( f^{-\alpha})$ is always generated at level $n-1$.

Consider  now a log resolution $\mu\colon Y \to X$ of the pair $(X,D)$ as above, and $E = (\mu^*D)_{\rm red}$. According 
to \cite[Theorem~10.1]{MP2}, the statement of Theorem \ref{main3} is equivalent to the following general form of local vanishing:

\begin{intro-corollary}\label{main4}
With the above notation, we have
$$R^i\mu_*\big(\Omega_Y^{n-i}(\log E)\otimes_{\shO_Y}\shO_Y(-\lceil  \mu^*\alpha D\rceil)\big)=0\quad\text{for}\quad i> n-\lceil\widetilde{\alpha}_f + \alpha\rceil.$$
\end{intro-corollary}

Recall for completeness that it is always the case that 
$$R^i\mu_*\big(\Omega_Y^j (\log E)\otimes_{\shO_Y}\shO_Y(-\lceil \mu^*\alpha D\rceil)\big)=0\quad\text{for}\quad 
i + j> n.$$
This is proved in \cite[Corollary~C]{MP2}, still using methods from the theory of mixed Hodge modules, but of 
a different flavor.

\noindent
{\bf Hodge ideals.}
The Hodge filtration on $\Mmod(f^{-\alpha})$ is best expressed and studied in terms of
the \emph{Hodge ideals} of $\alpha D$. According to \cite[\S4]{MP2}, for each $p \ge 0$ there is a coherent sheaf of ideals 
$I_p(\alpha D)$ on $X$ such that 
$$F_p \Mmod ( f^{-\alpha}) = I_p (\alpha D) \otimes \shO_X (pD) f^{-\alpha}.$$
Therefore Theorem \ref{main3} provides an effective bound describing which higher Hodge ideals of  $\alpha D$ are fully determined by lower ones. This type of result is very useful for concrete calculations of Hodge ideals, see \cite{MP1} and \cite{MP2}.

\begin{intro-corollary}
For every nonnegative integers $\ell$ and $p$, with $p \ge n - \lceil \widetilde{\alpha}_f + \alpha \rceil$, we have
$$F_\ell \Dmod_X\cdot\big(I_p(\alpha D)\otimes \shO_X(pD)f^{-\alpha}\big) =  I_{p +\ell}(\alpha D)\otimes \shO_X \big((p+\ell)D\big)f^{-\alpha}.$$
\end{intro-corollary}  

\noindent
{\bf Nearby and vanishing cycles.}
All the above results are consequences of a statement of independent interest regarding the generation level of the Hodge filtration on the graded quotients of the $V$-filtration associated to 
the regular function $f \in \shO_X(X)$.
Concretely, the $V$-filtration is defined on the the left $\Dmod_{X\times\CC}$-module $\iota_+\shO_X$, the push-forward of $\shO_X$ via the graph embedding 
$$\iota\colon X\hookrightarrow X\times\CC, \,\,\,\,\, x\mapsto \big(x,f(x)\big),$$
with respect to the hypersurface $\{t =0 \}$, where $t$ is the coordinate on $\CC$.
Recalling that this is a (discrete) decreasing filtration, we consider ${\rm Gr}^\alpha_V(\iota_+\shO_X) : = V^\alpha\iota_+\shO_X/ V^{> \alpha}\iota_+\shO_X$. These are $\Dmod_X$-modules that underlie Hodge modules  supported on the graph embedding of $X$; in particular they come endowed with a Hodge filtration $F_\bullet {\rm Gr}^\alpha_V(\iota_+\shO_X)$ induced by that  on $\iota_+\shO_X$. The cases $\alpha = 0$ and $\alpha \in (0,1]$ are intimately related to the vanishing, respectively nearby, cycles of $f$.  For details see \S\ref{scn:general} and \S\ref{scn:NV}. The main result we prove is:

\begin{intro-theorem}\label{NVC_generation}
If $f$ defines a singular, reduced hypersurface, and $\alpha \in [0,1]$ is a rational number, then the Hodge filtration on ${\rm Gr}^\alpha_V(\iota_+\shO_X)$ is generated at 
level $n - \lceil \widetilde{\alpha}_f + \alpha \rceil+1$.
\end{intro-theorem}

The proof of this theorem is the technical core of the paper. 
More precisely, we describe concretely the associated graded quotients of the Hodge filtration on these 
$\Dmod_X$-modules in the range below the minimal exponent of $f$; see Proposition \ref{filtration_nearby}. Using this, we apply a homological criterion for the generation level of the filtration on special filtered $\Dmod_X$-modules $(\Mmod, F)$ via the duality functor. This is proved in Proposition \ref{refined_criterion}, and is inspired by a duality approach to generation in \cite{Saito_microlocal}. In order to deduce Theorem \ref{main3} from Theorem \ref{NVC_generation}, the key tool is to reinterpret the main result of \cite{MP3} as a connection between the Hodge filtration on $\Mmod (f^{-\alpha})$ and the induced Hodge filtration on $V^\alpha$; see Proposition \ref{filtration_on_image}.

\noindent
{\bf Bounds in terms of singularity invariants in birational geometry.}
We conclude by noting that the minimal exponent $\widetilde{\alpha}_f$ can be bounded below in terms of basic invariants of the singularity, or in terms of discrepancies on a log resolution. This can be translated into bounds of a somewhat different flavor in the statements above.

Consider a log resolution $\mu\colon Y \to X$ of the pair $(X,D)$ as above, in the neighborhood of a (singular) point $x \in D$. Assuming in addition that the strict transform $\widetilde{D}$ of $D$ is smooth, we define integers $a_i$ and $b_i$ by the expressions
$$\mu^*D = \widetilde{D} + \sum_{i=1}^m a_iF_i\quad\text{and}\quad
K_{Y/X} = \sum_{i=1}^m b_iF_i,$$
where $F_1,\ldots,F_m$ are the prime exceptional divisors, and set 
$$\gamma : = \underset{i=1, \ldots, m}{\rm min} \left\{\frac{b_i +1}{a_i}\right\}.$$
Denote also by $d \ge 2$ the multiplicity of $D$ at $x$, and by $r$ the dimension of the singular locus of the projectivized tangent cone ${\mathbf P}(C_xD)$  (declaring that $r=-1$ if ${\mathbf P}(C_xD)$ is smooth). We then have the following lower bounds in a neighborhood of $x$:

$\bullet$\,\,\,\,\,\,$\widetilde{\alpha}_f \ge \gamma$. 

$\bullet$\,\,\,\,\,\,$\widetilde{\alpha}_f  \ge  \frac{n-r-1}{d}$.

\smallskip

The first is \cite[Corollary~D]{MP3} and the second is \cite[Theorem~E(3)]{MP3}. Note that, unlike 
$\widetilde{\alpha}_f$, $\gamma$ depends on the choice of log resolution. Finally, we also have:

$\bullet$~$k_0 : = \lfloor \widetilde{\alpha}_f - \alpha\rfloor$ is the $k$-log canonicity level of the pair $(X, \alpha D)$, according to \cite[Corollary~C]{MP3}.

We recall that $(X,\alpha D)$ is $0$-log canonical if it is log canonical, while being $k$-log canonical for $k\geq 1$ is a refinement of the statement that $D$ has rational singularities. It essentially means that the Hodge filtration on $\Mmod(f^{-\alpha})$ 
 is as simple as possible up to level $k$, namely equal to the pole order filtration; 
 the upshot of this paper is that this condition
also  imposes a bound on the generation level  of this Hodge filtration.

Further general properties of the minimal exponent $\widetilde{\alpha}_f$, and open problems, can be found in \cite[\S6]{MP3}.

\noindent
{\bf Acknowledgement.} 
We thank the referee for very useful comments that helped us improve the exposition.

\section{Preliminaries}

\subsection{Hodge filtration, $V$-filtration, and minimal exponent}\label{scn:general}
Let $X$ be a smooth $n$-dimensional complex algebraic variety and $f\in\shO_X(X)$ a nonzero regular function. 
Consider the graph embedding 
$$\iota\colon X\hookrightarrow X\times\CC, \,\,\,\,\, x\mapsto \big(x,f(x)\big)$$ 
and the left $\Dmod_{X\times\CC}$-module $\iota_+\shO_X$, as well as the corresponding right 
$\Dmod_{X\times\CC}$-module $\iota_+\omega_X$. A detailed discussion of the material in the paragraph below can be found for instance in \cite[\S2]{MP3}. We denote by $t$ the coordinate on $\CC$. Recall that we have $$\iota_+\shO_X \simeq \shO_X[t]_{f-t}/\shO_X[t],$$ 
with the obvious $\Dmod_{X\times\CC}$-module structure. Denoting by $\delta$ the class of $\frac{1}{f-t}$, every element in $\iota_+\shO_X$ can be written uniquely as
$$\sum_{i\geq 0}v_i \partial_t^i\delta,$$
with $v_i\in\shO_X$, only finitely many nontrivial. We clearly have the relation $t\delta=f\delta$.
With this description, multiplication by $t$ is given by
$$
t(v \partial_t^i\delta)=fv\partial_t^i\delta-iv\partial_t^{i-1}\delta
$$
and the action of a derivation $P\in {\rm Der}_{\CC}(\shO_X)$ is given by
$$P(v\partial_t^i\delta)=P(v)\partial_t^i\delta-P(f)v\partial_t^{i+1}\delta.$$
Recall also that the (trivial) Hodge filtration on $\shO_X$ induces a Hodge filtration on $\iota_+\shO_X$
given by 
\begin{equation}\label{Hodge_O}
F_{p+1} (\iota_+\shO_X) = \sum_{i=0}^p \shO_X \partial_t^i\delta
\end{equation}
(see, for example, \cite[(1.8.6)]{Saito-B}). 
We note that the shift by $1$ is needed in order to ensure compatibility when applying the convention for shifting filtrations as we pass from left to right filtered 
$\Dmod$-modules on $X$ and $X\times \CC$ respectively; see \S\ref{scn:NV}.

We next consider the rational $V$-filtration on $\iota_+ \shO_X$ with respect to $t$. Recall that this is an exhaustive, decreasing, discrete, and left continuous filtration
$(V^\alpha \iota_+ \shO_X)_{\alpha\in\QQ}$.
It is defined uniquely by a number of properties listed for instance in \cite[\S2]{MP3}.
The Hodge filtration on $\iota_+\shO_X$ induces a filtration on each $V^{\alpha}\iota_+\shO_X$ and thus the Hodge filtration
on ${\rm Gr}_V^{\alpha}(\iota_+\shO_X)=V^{\alpha}\iota_+\shO_X/V^{>\alpha}\iota_+\shO_X$. 

As is standard, we denote by $b_f (s)$ the Bernstein-Sato polynomial of $f$. Assuming that $D:={\rm div}(f)\neq 0$, the polynomial $(s+1)$ divides $b_f (s)$, and $\tilde{b}_f (s) = b_f (s) /(s+1)$ is the reduced Bernstein-Sato polynomial of $f$.  Following \cite{Saito-MLCT}, we denote by $\widetilde\alpha_f$ the negative of the largest root of $\tilde{b}_f (s)$. This is a positive rational number, and we use the convention that $\widetilde\alpha_f = \infty$ if $b_f (s)=s+1$, which happens precisely when $D$ is smooth. This invariant is called the \emph{minimal exponent} of $f$, see \cite{Saito-B}, and is a refined version of the  log canonical threshold of $f$, which is equal to ${\rm min}\{\widetilde\alpha_f, 1\}$. See \cite[\S6]{MP3} for a detailed discussion.

A crucial point is the following  link between the minimal exponent and  the $V$-filtration, combining the statements of \cite[Lemma~5.3]{MP3} and \cite[Corollary~6.1]{MP3}. 

\begin{lemma}\label{lem_char_min_exponent}
For an integer $p \ge 0$ and $\alpha\in (0,1]$, we have 
$$\partial_t^p \delta \in  V^\alpha\iota_+\shO_X \iff \widetilde{\alpha}_f \ge p + \alpha.$$
\end{lemma}

\medskip

For a $\QQ$-divisor $E$ on $C$, we denote by $\I (E)$ its multiplier ideal; see \cite[Chapter~9]{Lazarsfeld}. If $D =  {\rm div} (f)$, $\gamma >0$ is a rational number, and $E = \gamma D$, we will also 
use the notation $\I (f^\gamma)$ for $\I (E)$. The main result of \cite{BS} states that 
for every $\alpha>0$, we have
\begin{equation}\label{multiplier}
\I \big( f^{\alpha - \epsilon} \big) = V^\alpha\iota_+\shO_X \,\,\,\,\,\,{\rm for}\,\,\,\,0<\epsilon\ll 1.
\end{equation}

In order to define and study Hodge ideals for $\QQ$-divisors, in \cite{MP2} and \cite{MP3} we considered for each $\alpha > 0 $
 the twisted localization $\Dmod_X$-module
$$\Mmod (f^{-\alpha}) : = \shO_X (*D) f^{-\alpha},$$
with $D = {\rm div}(f)$, i.e. the free $\shO_X(*D)$-module of rank $1$ with generator the symbol $f^{-\alpha}$, with the action of derivations of $\shO_X$ given by 
$$P(wf^{-\alpha}) :=\left(P(w) - \alpha w\frac{ P(f)}{f}\right)f^{-\alpha}.$$
The $\Dmod_X$-module $\Mmod(f^{-\alpha})$  is a filtered direct summand of a $\Dmod_X$-module underlying a mixed Hodge module; see \cite[\S2]{MP2}. In particular, it is regular holonomic, with quasi-unipotent monodromy, and admits a Hodge filtration $F_p \Mmod(f^{-\alpha})$, with $p \ge 0$.
It is shown in \cite[\S4]{MP2} that if $Z$ is the support of $D$, then we can write
$$F_p \Mmod(f^{-\alpha}) = I_p (\alpha D) \otimes \shO_X (pZ) f^{-\alpha},$$
for an ideal $I_p (\alpha D)$, the \emph{$p$-th Hodge ideal of $\alpha D$}.

For every $\alpha \in \QQ$, we have an isomorphism of $\Dmod_X$-modules
\begin{equation}\label{identification}
\Mmod (f^{-\alpha}) \to \Mmod (f^{-\alpha -1}), \,\,\,\,\,\, w f^{-\alpha} \mapsto (wf) f^{-\alpha -1},
\end{equation}
which preserves the Hodge filtration; see \cite[\S2]{MP2}. As a special case, we naturally identify $\Mmod (f^{-1})$ with the usual localization $\shO_X(*D)$.
In particular, when $D$ is reduced and $\alpha=1$, this gives the Hodge ideals considered in \cite{MP1}.

An important input for this paper is the main result of \cite{MP3}, comparing the Hodge ideals and the $V$-filtration. We only state the case when 
$D={\rm div}(f)$ is reduced. We use the notation $Q_i(x)=\prod_{j=0}^{i-1}(x+j)$, with the convention that $Q_0=1$.

\begin{theorem}[{\cite[Theorem ~A$^\prime$]{MP3}}]\label{HI-Vfil}
If $f$ defines a reduced divisor $D$ and $\alpha$ is a positive rational number, then for every $p\geq 0$ we have
$$
I_p(\alpha D)=\left\{\sum_{j=0}^pQ_j(\alpha)f^{p-j}v_j ~|~ \sum_{j=0}^pv_j\partial_t^j\delta\in V^{\alpha}\iota_+\shO_X\right\}.
$$
\end{theorem}

\subsection{Nearby and vanishing cycles}\label{scn:NV}
Later on we will need bounds for the generation level of the Hodge filtration on nearby and vanishing 
cycles. To this end we will make use of the duality functor ${\mathbf D}$ on filtered $\Dmod$-modules \cite[\S2.4]{Saito-MHP}.
In order to apply duality, we will pass to the corresponding right $\Dmod_X$-modules.

We recall that there is an equivalence of categories between filtered left and right $\Dmod_X$-modules. Given a filtered left $\Dmod_X$-modules
$(\Mmod,F)$, we denote by $(\Mmod^r,F)$ the corresponding filtered right $\Dmod_X$-module. At the level of $\shO_X$-modules we have
$\Mmod^r=\omega_X\otimes_{\shO_X}\Mmod$, while the filtration on $\Mmod^r$ is given by
$$F_{p-n}\Mmod^r=\omega_X\otimes_{\shO_X}F_p\Mmod\quad\text{for all}\quad p\in\ZZ,$$
where $n=\dim(X)$.

For right $\Dmod_X$-modules it is customary to use the increasing $V$-filtration. This is related to the $V$-filtration on the corresponding 
left $\Dmod_X$-module as follows. If $\Mmod$ is a left $\Dmod_{X\times\CC}$-module and we consider the $V$-filtrations with respect to the coordinate $t$
on $\CC$, then
$$V_{\alpha}\Mmod^r=\omega_{X\times\CC}\otimes_{\shO_{X\times\CC}}V^{-\alpha}\Mmod=\omega_X\otimes_{\shO_X}V^{-\alpha}\Mmod,$$
where we identify in the obvious way $\omega_{X\times\CC}$ with the pull-back of $\omega_X$.

It is also customary, for a filtered $\Dmod_X$-module $(\Mmod,F)$ and an integer $q$, to denote 
$(\Mmod,F)(q)=\big(\Mmod,F[q]\big)$, with 
$$F[q]_p\Mmod=F_{p-q}\Mmod \,\,\,\,\,\,{\rm  for~all} \,\,\,\, p\in\ZZ.$$

Let now $(\Mmod,F)$ be the filtered right $\Dmod_{X\times\CC}$-module underlying a pure polarizable Hodge module of weight $d$. Recall that the polarization induces an isomorphism ${\mathbf D}(\Mmod,F)\simeq (\Mmod,F)(d)$. The nearby and vanishing cycles of $(\Mmod,F)$ with respect to $t$ are given, respectively,  by 
$$\Psi_t(\Mmod,F)=\bigoplus_{-1\leq\beta<0}\big({\rm Gr}_{\beta}^V(\Mmod),F\big)(1)\quad {\rm and} \quad  \Phi_{t,1}(M,F)=\big({\rm Gr}_0^V(\Mmod),F\big).$$
We also use the notation $\Psi_{t,\beta}(\Mmod,F)$ for $\big({\rm Gr}_{\beta}^V(\Mmod),F\big)(1)$, when $\beta \in (-1, 0)$, but $\Psi_{t,1}(\Mmod,F)$ for $\big({\rm Gr}^V_{-1}(\Mmod),F\big)(1)$.

It is a general fact that the duality functor commutes with nearby and vanishing cycles. 
The results that follow can be found in \cite[Theorem~1.6]{Saito_duality}. Concretely, we have
canonical isomorphisms 
$${\mathbf D}\Psi_t(\Mmod,F)(1)\simeq \Psi_t{\mathbf D}(\Mmod,F)\quad\text{and}\quad
{\mathbf D}\Phi_{t,1}(\Mmod,F)\simeq \Phi_{t,1}{\mathbf D}(\Mmod,F).$$
Using the fact that ${\mathbf D}(\Mmod,F)\simeq \big(\Mmod,F\big)(d)$,
we obtain isomorphisms
$${\mathbf D}\Psi_t(\Mmod,F)\simeq \Psi_t(\Mmod,F)(d-1)\quad\text{and}\quad {\mathbf D}\Phi_{t,1}(\Mmod,F)\simeq \Phi_{t,1}(\Mmod,F)(d).$$
We can in fact be more precise about the first of these isomorphisms; 
there is a canonical isomorphism
\begin{equation}\label{eq0_dual_nearby_cycles}
{\mathbf D}\Psi_{t,1}(\Mmod,F)\simeq \Psi_{t,1}(\Mmod,F)(d-1)
\end{equation}
and
for every $\beta\in (-1,0)$, there is a canonical isomorphism
\begin{equation}\label{eq_dual_nearby_cycles}
{\mathbf D}\Psi_{t,\beta}(\Mmod,F)\simeq \Psi_{t,-\beta-1}(\Mmod,F)(d-1).
\end{equation}

In what follows, we will only be interested in the case when $(\Mmod,F)$
is the filtered right $\Dmod_{X\times\CC}$-module 
$(\iota_+\omega_X,F)$ corresponding to $(\iota_+\shO_X,F)$. Note that in this case we have $d=n$,
hence the isomorphism (\ref{eq0_dual_nearby_cycles}) gives
\begin{equation}\label{eq2_dual_nearby_cycles}
{\mathbf D}{\rm Gr}_{-1}^V(\iota_+\omega_X)\simeq {\rm Gr}^V_{-1}(\iota_+\omega_X)(1+n)
\end{equation}
while the isomorphism (\ref{eq_dual_nearby_cycles}) gives
\begin{equation}\label{eq3_dual_nearby_cycles}
{\mathbf D}{\rm Gr}_{\beta}^V(\iota_+\omega_X)\simeq {\rm Gr}^V_{-1-\beta}(\iota_+\omega_X)(1+n)
\quad\text{for every}\quad \beta\in (-1,0).
\end{equation}
Similarly, we have
\begin{equation}\label{eq_vanishing_cycles}
{\mathbf D}{\rm Gr}_0^V(\iota_+\omega_X)\simeq {\rm Gr}_0^V(\iota_+\omega_X)(n).
\end{equation}

Finally, we note that since the Hodge filtration on ${\rm Gr}_{\beta}^V(\iota_+\omega_X)$ is induced by that on $\iota_+\omega_X$,
which is the filtered right $\Dmod_{X\times\CC}$-module corresponding to $\iota_+\shO_X$, using the convention above on upper and lower indexed $V$-filtrations we have 
\begin{equation}\label{eq_comp_V_filtrations}
F_{p-n-1}{\rm Gr}_{\beta}^V(\iota_+\omega_X)=\omega_X\otimes_{\shO_X}F_p{\rm Gr}_V^{-\beta}(\iota_+\shO_X).
\end{equation}

\subsection{Generation level}
Let $(\Mmod,F)$ be a right $\Dmod_X$-module with a good filtration. The filtration $F$ is generated at level $q$ if
$$F_q\Mmod \cdot F_{\ell}\Dmod_X =F_{q+\ell}\Mmod \quad\text{for all}\quad \ell\geq 0,$$
or equivalently
$$F_p\Mmod \cdot F_1\Dmod_X=F_{p+1}\Mmod \quad\text{for all}\quad p\geq q.$$
A similar definition holds for left $\Dmod_X$-modules, as in the introduction.
Note that such $q$ always exists by the definition of a good filtration. Another interpretation is that the filtration is generated at level $q$ if and only if ${\rm Gr}^F_{\bullet} \Mmod$ is generated  in degrees $\leq q$ as a graded module over
$$\mathcal{A}_X:={\rm Gr}^F_{\bullet} \Dmod_X \simeq {\rm Sym}_{\shO_X}^{\bullet}\shT_X,$$
where  $\shT_X$ is the tangent sheaf of $X$.

A generation criterion using the duality functor is given by the following result; see \cite[Lemma~2.5]{Saito_microlocal} and its proof.

\begin{proposition}\label{criterion_via_dual}
If $(\Mmod,F)$ is a filtered right $\Dmod_X$-module underlying a mixed Hodge module such that $F_{-q-1}{\mathbf D}(\Mmod)=0$, then the filtration
on $\Mmod$ is generated at level $q$.
\end{proposition}

We will also need a refinement of this criterion for (essentially) self-dual $(\Mmod, F)$, and for this we formulate more precisely the setup provided by duality. 
The $0$-section of the cotangent bundle corresponds to a surjective morphism $\mathcal{A}_X\to\shO_X$.
 We denote by $K^{\bullet}$ the corresponding Koszul complex 
$$0\to K^{-n}\to \cdots \to K^{-1}\to K^0=\shO_X\to 0$$
placed in degrees $-n,\ldots,0$, where $K^{-i}=\wedge^{i}\shT_X\otimes_{\shO_X}\mathcal{A}_X(i)$.
Note that we use the opposite of the standard convention for degree-shift, namely ${\mathcal P}(i)_m={\mathcal P}_{m-i}$. 
This is a complex of graded free $\mathcal{A}_X$-modules, which gives a free resolution of $\shO_X$ as an $\mathcal{A}_X$-module. 

Suppose now that $(\Mmod,F)$ is a filtered \emph{right} $\Dmod_X$-module that underlies a mixed Hodge module. In this case we have that  ${\rm Gr}_{\bullet}^F \Mmod$ is a Cohen-Macaulay $\mathcal{A}_X$-module by \cite[Lemme~5.1.13]{Saito-MHP} (and, more generally, one can consider 
filtered $\Dmod_X$-modules with this property). 
Recall from \cite[\S2.2]{Saito-MHP} that $\widetilde{\rm DR}(\Mmod,F)$ is the 
filtered differential complex
$$0\to \Mmod\otimes_{\shO_X}\wedge^n\shT_X\to\cdots\to \Mmod\otimes_{\shO_X}\shT_X\to \Mmod\to 0,$$
placed in degrees $-n,\ldots,0$, such that the level $p$ part is given by
$$0\to F_{p-n}\Mmod\otimes_{\shO_X}\wedge^n\shT_X\to\cdots\to F_{p-1}\Mmod\otimes_{\shO_X}\shT_X\to F_p\Mmod\to 0.$$
The maps are \emph{not} $\shO_X$-linear, but by taking the associated graded objects, we obtain complexes of $\shO_X$-modules. More precisely, we have
$${\rm Gr}_p^F\widetilde{\rm DR}(\Mmod,F)\simeq (P\otimes_{\mathcal{A}_X}K^{\bullet})_p,$$
where $P={\rm Gr}^F_{\bullet} \Mmod$. Note that $P\otimes_{\mathcal{A}_X}K^{\bullet}$ represents the object $P\overset{\derL}\otimes_{\mathcal{A}_X}
\shO_X$ in the derived category of graded $\shO_X$-modules. 

An important feature of the duality functor is the following isomorphism in the derived category of filtered differential complexes of $\shO_X$-modules:
$${\mathbf D}\big(\widetilde{\rm DR}(\Mmod,F)\big)\simeq \widetilde{\rm DR}\big({\mathbf D}(\Mmod,F)\big),$$
having the property:
$${\rm Gr}_p^F{\mathbf D}\big(\widetilde{\rm DR}(\Mmod,F)\big)\simeq \derR\sheafhom_{\shO_X}\big({\rm Gr}_{-p}^F\widetilde{\rm DR}(\Mmod,F), 
\omega_X[n]\big)\quad\text{for all}\quad
p\in\ZZ.$$
See \cite[\S2.4]{Saito-MHP}, and also \cite[Remark 2.6]{Saito_microlocal}.

Suppose now that $(\Mmod,F)$ satisfies ${\mathbf D}(\Mmod, F)\simeq (\Mmod, F)(d)$ for some $d\in \ZZ$; this is for instance the case for the nearby and vanishing cycle modules in the previous section.  By combining the above facts, we see that for every $p\in\ZZ$ we have an isomorphism in the derived category of $\shO_X$-modules:
\begin{equation}\label{eq_isom}
(P\otimes_{\mathcal{A}_X}K^{\bullet})_{p-d}\simeq\derR\sheafhom_{\shO_X}\big((P\otimes_{\mathcal{A}_X}K^{\bullet})_{-p},\omega_X\big)[n].
\end{equation}


Denoting $A^{\bullet}:=P\otimes_{\mathcal{A}_X}K^{\bullet}$, using the discussion at the beginning of the section 
we see that the filtration on $\Mmod$ is generated at level $q$ if and only if ${\mathcal H}^0(A^{\bullet})_p=0$
for every $p>q$. The isomorphism (\ref{eq_isom}) gives
$${\mathcal H}^0(A^{\bullet})_p\simeq {\mathscr Ext}_{\shO_X}^n(A^{\bullet}_{-p-d},\omega_X).$$
On the other hand, we have the first-quadrant spectral sequence 
$$E_1^{i,j}={\mathscr Ext}_{\shO_X}^j(A^{-i}_{-p-d},\omega_X) \Rightarrow {\mathscr Ext}_{\shO_X}^{i+j}(A^{\bullet}_{-p-d},\omega_X).$$
Recall also that by definition, we have
$$A^{-i}_{-p-d}={\rm Gr}_{-p-d-i}^F \Mmod \otimes_{\shO_X}\wedge^i\shT_X.$$
Thus for such filtered $\Dmod_X$-modules we obtain the following refinement of the criterion in Proposition \ref{criterion_via_dual}: 

\begin{proposition}\label{refined_criterion}
If $(\Mmod, F)$ underlies a mixed Hodge module and ${\mathbf D}(\Mmod, F)\simeq (\Mmod, F)(d)$, then the filtration on $\Mmod$ is generated at level $q$ if 
$${\mathscr Ext}_{\shO_X}^j({\rm Gr}_{-p-d-n + j}^F \Mmod\otimes_{\shO_X}\wedge^i\shT_X,\omega_X) = 0 \,\,\,\,\,\,{\rm for~all} \,\,\,\,\,\,0 \le j \le n,$$
for every $p > q$.
\end{proposition}

\section{Main results}

We continue to work on a smooth complex variety $X$, endowed with a nonzero regular function $f\in \shO_X (X)$.
We use the notation of the previous section.

\subsection{Generation level for ${\rm Gr}^\alpha_V(\iota_+\shO_X)$}
We start by proving the key Theorem \ref{NVC_generation}; this is
 split here into Propositions~\ref{crit1}, \ref{crit0} and \ref{crit2}, the last being the most involved.
We begin with a generation bound for ${\rm Gr}_V^{\alpha}(\iota_+\shO_X)$ with $\alpha\in (0,1)$. This case only needs the criterion in Proposition~\ref{criterion_via_dual}.

\begin{proposition}\label{crit1}
For $\alpha\in (0,1)$ and $q\geq 1$, the Hodge filtration on ${\rm Gr}_V^{\alpha}(\iota_+\shO_X)$ 
is generated at level $q$ if $F_{n-q}{\rm Gr}_V^{1-\alpha}(\iota_+\shO_X)=0$.
In particular, if $f$ defines a singular hypersurface, then the Hodge filtration on 
${\rm Gr}_V^{\alpha}(\iota_+\shO_X)$ is generated at level $n-\lceil\widetilde{\alpha}_f+\alpha\rceil+1$. 
\end{proposition}

\begin{proof}
It follows from (\ref{eq_comp_V_filtrations}) that the filtration on ${\rm Gr}_V^{\alpha}(\iota_+\shO_X)$ is generated at level $q$
if and only if the filtration on ${\rm Gr}^V_{-\alpha}(\iota_+\omega_X)$ is generated at level $q-n-1$.
Using the isomorphism (\ref{eq3_dual_nearby_cycles}), we deduce in turn from 
Proposition~\ref{criterion_via_dual} that this is the case if
$$F_{n-q}{\rm Gr}^V_{\alpha-1}(\iota_+\omega_X)(n+1)=F_{-q-1}{\rm Gr}^V_{\alpha-1}(\iota_+\omega_X)$$
is $0$. The latter condition is equivalent with $F_{n-q}{\rm Gr}_V^{1-\alpha}(\iota_+\shO_X)=0$ by another application of  (\ref{eq_comp_V_filtrations}),
giving the first assertion in the proposition.

For the second assertion, note that by Lemma~\ref{lem_char_min_exponent}, for every $j\geq 0$ 
and every $\beta\in (0,1)$ we have the equivalence
$$\partial_t^j\delta\in V^{\beta} \iff \widetilde{\alpha}_f\geq j+\beta.$$ 
In particular, if this holds for $j\geq 1$, it also holds for $j-1$. 
If $q=n-\lceil\widetilde{\alpha}_f+\alpha\rceil+1$, then $q> n-\widetilde{\alpha}_f-\alpha$,
and we conclude that there is
$\beta$ with $1-\alpha<\beta<1$, such that $\partial_t^{n-q-1}\delta\in V^{\beta}\iota_+\shO_X$.
In this case we have 
$F_{n-q}V^{\beta}\iota_+\shO_X=F_{n-q}\iota_+\shO_X$, hence clearly $F_{n-q}{\rm Gr}_V^{1-\alpha}(\iota_+\shO_X)=0$. 
\end{proof}

A similar proof works for $\alpha =0$; we include it for completeness, even though this is not relevant for 
the rest of the paper.

\begin{proposition}\label{crit0}
If $F_{n-q+1}{\rm Gr}_V^0(\iota_+\shO_X)=0$ for some $q\geq 1$, then
the Hodge filtration on ${\rm Gr}_V^0(\iota_+\shO_X)$ is generated at level $q$.
In particular, if $f $defines a singular hypersurface, then the Hodge filtration on ${\rm Gr}_V^0(\iota_+\shO_X)$
is generated at level $n-\lceil\widetilde{\alpha}_f\rceil+1$. 
\end{proposition}
\begin{proof}
Arguing as above, using (\ref{eq_vanishing_cycles}) and Proposition~\ref{criterion_via_dual} we see that the Hodge filtration on ${\rm Gr}^0_V(\iota_+\shO_X)$
 is generated at level
$q$ if $F_{n-q+1}{\rm Gr}_V^0(\iota_+\shO_X)=0$. This in turn holds if $q=n-\lceil\widetilde{\alpha}_f\rceil+1$, since
Lemma \ref{lem_char_min_exponent} implies that there exists $\beta > 0$ such that $\partial_t^{\lceil \widetilde{\alpha}_f\rceil -1}\delta\in V^{\beta}$.
\end{proof}

\medskip

For ${\rm Gr}_V^1(\iota_+\shO_X)$ we need to use a more refined argument. We start by specializing the criterion in Proposition \ref{refined_criterion} to the 
$\Dmod_{X\times \CC}$-module 
$\Mmod={\rm Gr}^V_{-1}(\iota_+\omega_X)$, in which case we have $d=n+1$ by (\ref{eq2_dual_nearby_cycles}), so that
the vanishing in the proposition concerns
$${\mathscr Ext}_{\shO_X}^j\big({\rm Gr}_{j-p-2n-1}^F{\rm Gr}^V_{-1}(\iota_+\omega_X)\otimes_{\shO_X} \wedge^i\shT_X ,\omega_X\big)$$ 
$$\simeq {\mathscr Ext}_{\shO_X}^j\big({\rm Gr}_{j-p-n}^F{\rm Gr}_V^1(\iota_+\shO_X),\shO_X)\otimes_{\shO_X} \Omega_X^i.$$
Furthermore, the filtration on ${\rm Gr}_V^1(\iota_+\shO_X)$ is generated at level $q$ if and only if the filtration on 
${\rm Gr}^V_{-1}(\iota_+\omega_X)$ is generated at level $q-n-1$. We thus obtain

\begin{corollary}\label{Ext_criterion}
The Hodge filtration on ${\rm Gr}_V^1(\iota_+\shO_X)$ is generated at level $q$ if 
$${\mathscr Ext}_{\shO_X}^j\big({\rm Gr}_{j - p}^F {\rm Gr}_V^1(\iota_+\shO_X), \shO_X\big) = 0 
\,\,\,\,\,\,{\rm for ~all} \,\,\,\, 0\leq j\leq n \,\,\,\,{\rm and} \,\,\,\,p>q-1.$$
\end{corollary}

To apply this criterion, we need a better understanding of the terms ${\rm Gr}^F_k{\rm Gr}^1_V(\iota_+\shO_X)$.
To this end, for every $k\geq 0$ we introduce the following coherent ideals of $\shO_X$:
$$J_k=\{h\in\shO_X\mid h\partial_t^k\delta\in V^1\iota_+\shO_X\}\quad\text{and}\quad
J'_k=\{h\in \shO_X\mid h\partial_t^k\delta\in V^{>1}\iota_+\shO_X\}.$$
From now on, we will only deal with the $V$-filtration on $\iota_+\shO_X$, hence in order to simplify the notation
we often denote $V^{\alpha}=V^{\alpha}\iota_+\shO_X$ and ${\rm Gr}^{\alpha}_V={\rm Gr}^{\alpha}_V(\iota_+\shO_X)$.

We will make use of the fact that $J'_k\subseteq (f)$ for all $k\geq 0$. 
In fact, we prove the following more precise result:

\begin{lemma}\label{ideals_J'}
If $f$ defines a reduced hypersurface, then for every $k\geq 0$, we have $J'_k=(f^{k+1})$.
\end{lemma}

\begin{proof}
It is well known that ${\mathcal I}(f^{k+1})=(f^{k+1})$, and so by ($\ref{multiplier}$) 
it follows that $f^{k+1}\delta\in V^{>(k+1)}$.
We thus have $f^{k+1}\partial_t^k\delta\in V^{>1}$, hence $f^{k+1}\in J'_k$.

It suffices to prove the reverse inclusion $J'_k\subseteq (f^{k+1})$ on 
an open subset $U$ of $X$ such that ${\rm codim}_X(X\smallsetminus U)\geq 2$. Since $f$
defines a reduced hypersurface, we can find such a subset $U$ on which $f$ is smooth. 
We will therefore assume from now on that ${\rm div}(f)$ is smooth. After passing to a suitable
open cover of $X$, we may further assume that we have an algebraic system of coordinates 
$x_1,\ldots,x_n$ such that $f=x_1$. 

Recall that in this case the $V$-filtration on $\iota_+ \shO_X$ only jumps at integers (hence $V^{>1}=V^2$) and for every $m\geq 1$,
$V^m$ is generated over $\Dmod_X$ by $x_1^{m-1}$. This follows easily by checking that this definition
satisfies the defining properties of the $V$-filtration. (For a more general statement valid for arbitrary simple normal crossing divisors, see \cite[Theorem~3.4]{Saito-MHM}.) In particular, we see that $V^2$ is generated as an $\shO_X$-module
by $\partial_{x_1}^ix_1\delta$, for $i\geq 0$. Since $\partial_{x_1}^i\delta=(-1)^i\partial_t^i\delta$, we have
$$\partial_{x_1}^ix_1\delta=x_1\partial_{x_1}^i\delta+[\partial_{x_1}^i,x_1]\delta=(-1)^ix_1\partial_t^i\delta+(-1)^{i-1}i\partial_t^{i-1}\delta.$$
We conclude that given a regular function $h$, we have $h\partial_t^k\delta\in V^2$ if and only if there are
regular functions $g_0,\ldots,g_p$ such that
$$h\partial_t^k\delta=\sum_{i=0}^pg_i\partial_{x_1}^ix_1\delta=g_0x_1\delta+\sum_{i=1}^p(-1)^ig_i(x_1\partial_t^i\delta-i\partial_t^{i-1}\delta).$$
This equality holds if and only if $g_i=0$ for $i>k$, $h=(-1)^kx_1g_k$, and 
$$x_1g_i+(i+1)g_{i+1}=0\quad\text{for}\quad 0\leq i\leq k-1.$$
This clearly implies that $h\in (x_1^{k+1})$, completing the proof of the lemma.
\end{proof}

We are now able to establish the connection between the Hodge filtration  on ${\rm Gr}^1_V$ and the minimal exponent $\widetilde{\alpha}_f$.

\begin{proposition}\label{filtration_nearby}
If $f$ defines a reduced hypersurface and $p \ge 0$ is an integer such that $\widetilde{\alpha}_f > p$, then 
$${\rm Gr}_{p+1}^F{\rm Gr}_V^1(\iota_+\shO_X) \simeq J_p/(f) \,\,\,\,\,\,  {\rm and} \,\,\,\,\,\,{\rm Gr}_{i+1}^F{\rm Gr}_V^1(\iota_+\shO_X) \simeq \shO_X/(f)\,\,\,\,{\rm for}\,\,\,\,0 \le i\le p-1$$
${\rm (}$note that the second statement is vacuous for $p = 0$${\rm )}$.
\end{proposition}
\begin{proof}
Fix $0 \le k \le p$. Since $k <\widetilde{\alpha}_f$, it follows from Lemma~\ref{lem_char_min_exponent}
that $\partial_t^i\delta\in V^{>0}$ for $0\leq i\leq k$. This implies that for every such $i$, we have
$t\partial_t^i\delta\in V^{>1}$.
Note that
$$t\partial_t^i\delta=f\partial_t^i\delta-i\partial_t^{i-1}\delta\quad\text{for}\quad 1\leq i\leq k,$$
hence $t\partial_t\delta,\ldots,t\partial_t^k,\partial_t^k\delta$ give a basis of $F_{k+1}\iota_+\shO_X$ over 
$\shO_X$.
Since all but the last one of these elements lie in $F_{k+1}V^{>1}$,  we have
 a canonical isomorphism
 \begin{equation}\label{eq2_gen_level_Gr_1}
F_{k+1}{\rm Gr}_V^1=F_{k+1}V^1/F_{k+1}V^{>1}\simeq J_k/J'_k.
\end{equation}

If $k\leq p-1$, then $\partial_t^k\delta\in V^1$ by Lemma~\ref{lem_char_min_exponent}, hence $J_k=\shO_X$. 
Moreover, via the isomorphisms (\ref{eq2_gen_level_Gr_1}), the inclusion
$$F_{k+1}{\rm Gr}_V^1\to F_{k+2}{\rm Gr}_V^1$$
maps the class of $1$ in $\shO_X/J'_{k}$ to the class of $\frac{1}{k+1}f$ in $J_{k+1}/J'_{k+1}$.
Indeed, this follows from the fact that 
$$\partial_t^k\delta=\frac{1}{k+1}f\partial_t^{k+1}\delta-\frac{1}{k+1}t\partial_t^{k+1}\delta.$$
We thus conclude that 
$${\rm Gr}^F_{k+2}{\rm Gr}^1_V\simeq J_{k+1}/\big(J'_{k+1}+(f)\big)=J_{k+1}/(f),$$
where the equality follows from Lemma~\ref{ideals_J'}.
Furthermore, as we have already mentioned, if $k\leq p-2$, then $J_{k+1}=\shO_X$, hence
${\rm Gr}^F_{k+2} {\rm Gr}^1_V\simeq \shO_X/(f)$.

On the other hand, note that we always have
$${\rm Gr}^F_1{\rm Gr}_V^1=F_1{\rm Gr}_V^1\simeq J_0/J'_0=J_0/(f),$$
where the last equality holds by Lemma~\ref{ideals_J'}. Furthermore, $J_0=\shO_X$ if $p\geq 1$. 
This completes the proof of the proposition.
\end{proof}

\begin{proposition}\label{crit2}
If $f$ defines a singular, reduced hypersurface, then the
 Hodge filtration on ${\rm Gr}_V^1(\iota_+\shO_X)$ is generated at level $n- \lceil \widetilde{\alpha}_f \rceil$.
\end{proposition}

\begin{proof}
Equivalently, we need to check that if $p$ is a nonnegative integer such that $\widetilde{\alpha}_f>p$, then the filtration on ${\rm Gr}_V^1$ is generated at level $n-1-p$. (Note that since $f$ defines a singular hypersurface, we have $\widetilde{\alpha}_f\leq\frac{n}{2}$ as mentioned in the introduction, hence our assumption on $p$
implies $n-1-p\geq 1$.) It follows then from Corollary  \ref{Ext_criterion} that it is enough to show: 
\begin{equation}\label{eq1_gen_level_Gr_1}
{\mathscr Ext}_{\shO_X}^j({\rm Gr}^F_{j-i}{\rm Gr}^1_V,\shO_X)=0\quad\text{for}\quad 0\leq j\leq n\quad\text{and}\quad i>n-2-p.
\end{equation}
Note that we only need to consider $i$ and $j$ such that $0\leq j-i-1\leq n-i-1\leq p$.

To see this, we use the isomorphisms in Proposition \ref{filtration_nearby}.
First, the short exact sequence
$$0\longrightarrow \shO_X\overset{\cdot f}\longrightarrow \shO_X\longrightarrow \shO_X/(f)\longrightarrow 0$$
gives ${\mathscr Ext}_{\shO_X}^m\big(\shO_X/(f),\shO_X\big)=0$ for all $m\geq 2$. We thus see that if $0\leq j-i-1\leq p-1$, we have 
$${\mathscr Ext}_{\shO_X}^j({\rm Gr}^F_{j-i}{\rm Gr} _V^1,\shO_X)\simeq {\mathscr Ext}_{\shO_X}^j\big(\shO_X/(f),\shO_X\big)=0,$$
since $j\geq i+1\geq n-p\geq 2$. On the other hand, if $j-i-1=p$, then $j=n$, and
 the short exact sequence
$$0\to J_p/(f)\to \shO_X/(f)\to \shO_X/J_p\to 0$$
implies that 
$${\mathscr Ext}_{\shO_X}^n\big({\rm Gr}^F_{p+1}{\rm Gr}^1_V,\shO_X\big)\simeq {\mathscr Ext}_{\shO_X}^n\big(J_p/(f),\shO_X\big)$$ 
is a quotient of ${\mathscr Ext}_{\shO_X}^n\big(\shO_X/(f),\shO_X\big)=0$. This completes the proof of the proposition.
\end{proof}

\begin{remark}
In the statements of Propositions~\ref{crit1}, \ref{crit0}, and \ref{crit2}, we assumed that the hypersurface defined by $f$ is singular, in order to avoid the case
when $\widetilde{\alpha}_f=\infty$. If $f$ defines a smooth hypersurface, then ${\rm Gr}_V^{\alpha}$ is nonzero only when $\alpha$ is an integer and 
the Hodge filtration on both
${\rm Gr}_V^0$ and ${\rm Gr}_V^1$ is generated in level $0$.
\end{remark}

\subsection{The Hodge filtrations on $V^{\alpha}$ and ${\mathcal M}(f^{-\alpha})$}
Let $\pi\colon X\times\CC\to X$ be the projection onto the first component. Given $\alpha\in\QQ$, we consider the map
$$\tau_{\alpha}\colon \pi_*V^{\alpha}\iota_+\shO_X\to {\mathcal M}(f^{-\alpha})$$
given by
$$\tau_{\alpha}\left(\sum_{i=0}^pv_i\partial_t^i\delta\right)=\left(\sum_{i=0}^pQ_i(\alpha)\frac{v_i}{f^i}\right)f^{-\alpha},$$
where $Q_i(x)=\prod_{j=0}^{i-1}(x+j)$ (with the convention that $Q_0=1$). Note that both sides have $\Dmod_X$-module structure; in fact 
$\pi_*V^\alpha\iota_+\shO_X$ is naturally a $\Dmod_X[t, \partial_t t]$-module.

\begin{lemma}\label{lem1}
The map $\tau_{\alpha}$ is a morphism of $\Dmod_X$-modules. Moreover, we have
\begin{equation}\label{eq1_lem1}
\tau_{\alpha+1}(tv)=\tau_{\alpha}(v)\quad \text{for every}\quad v\in \pi_*V^{\alpha}\iota_+\shO_X\quad\text{and}
\end{equation}
\begin{equation}\label{eq2_lem1}
\tau_{\alpha}(\partial_tv)=\alpha\cdot \tau_{\alpha+1}(v)\quad\text{for every}\quad v\in \pi_*V^{\alpha+1}\iota_+\shO_X,
\end{equation}
where the equalities hold via the identification in ${\rm (}$$\ref{identification}$${\rm )}$. 
\end{lemma}

\begin{proof}
We may and will assume that $X$ is affine.
The fact that $\tau_{\alpha}(gu)=g\cdot\tau_{\alpha}(u)$ for every $g\in\shO_X(X)$ and every global section $u$ of $V^{\alpha}$ is clear. 
Suppose now that $v=\sum_{i=0}^pv_i\partial_t^i\delta\in V^{\alpha}$ and $P$ is a $\CC$-derivation of $\shO_X(X)$. We have
$$Pv=\sum_{i=0}^pP(v_i)\partial_t^i\delta-\sum_{i=0}^pP(f)v_i\partial_t^{i+1}\delta,$$
hence 
$$\tau_{\alpha}(Pv)=\left(\sum_{i=0}^pQ_i(\alpha)\frac{P(v_i)}{f^i}-\sum_{i=0}^pQ_{i+1}(\alpha)\frac{v_iP(f)}{f^{i+1}}\right)f^{-\alpha}$$
$$=P \left( \left(\sum_{i=0}^pQ_i(\alpha)\frac{v_i}{f^i}\right)f^{-\alpha}\right)=P (\tau_{\alpha}(v)),$$
where we used the fact that $Q_{i+1}(\alpha)=(\alpha+i)Q_i(\alpha)$ and
$$P \left(\frac{h}{f^i}f^{-\alpha}\right)=\frac{P(h)}{f^i}f^{-\alpha}-\frac{(\alpha+i)hP(f)}{f^{i+1}}f^{-\alpha}.$$

By the definition of the $V$-filtration, if $v \in V^\alpha$, then $tv\in V^{\alpha+1}$ (and for $\alpha > 0$,
multiplication by $t$ induces an isomorphism of $\Dmod_X$-modules $V^\alpha \to V^{\alpha +1}$).
In order to prove (\ref{eq1_lem1}), note first that if $v=\sum_{i=0}^pv_i\partial_t^i\delta$, then
$$tv=\sum_{i=0}^pfv_i\partial_t^i\delta-\sum_{i=1}^piv_i\partial_t^{i-1}\delta.$$
We thus have
$$\tau_{\alpha+1}(tv)=\left(\sum_{i=0}^pQ_i(\alpha+1)\frac{fv_i}{f^i}-\sum_{i=1}^pQ_{i-1}(\alpha+1)\frac{iv_i}{f^{i-1}}\right)f^{-\alpha-1}.$$
Since
$$Q_i(\alpha+1)-iQ_{i-1}(\alpha+1)=Q_i(\alpha)\quad\text{for}\quad i\geq 1$$
and $Q_0(\alpha+1)=Q_0(\alpha)$, we conclude that $\tau_{\alpha+1}(tv) = 
\tau_{\alpha} (v)$ via ($\ref{identification}$).

Suppose now that $v=\sum_{i=0}^pv_i\partial_t^i\delta\in V^{\alpha+1}$, hence
$\partial_t v=\sum_{i=0}^pv_i\partial_t^{i+1}\delta \in V^\alpha$. We then have
$$\tau_{\alpha}(\partial_tv)=\left(\sum_{i=0}^pQ_{i+1}(\alpha)\frac{v_i}{f^{i+1}}\right)f^{-\alpha}=
\alpha\cdot\left(\sum_{i=0}^pQ_i(\alpha+1)\frac{v_i}{f^i}\right)f^{-\alpha-1}=\alpha\cdot\tau_{\alpha+1}(v),$$
which proves (\ref{eq2_lem1}).
\end{proof}

\begin{proposition}\label{filtration_on_image}
If $D = {\rm div}(f)$ is a reduced divisor, then for every $\alpha>0$ the morphism $\tau_{\alpha}$ is surjective, and the Hodge filtration on the image is, up to a shift by $1$, the induced filtration from that on $V^{\alpha}\iota_+\shO_X$. More precisely, we have
$$F_p{\mathcal M}(f^{-\alpha})=\tau_{\alpha}(F_{p+1}V^{\alpha}\iota_+\shO_X)\quad\text{for all}\quad p\geq 0.$$
\end{proposition}

\begin{proof}
Thanks to ($\ref{Hodge_O}$), the elements of $F_{p+1}V^{\alpha}$ are the sums $\sum_{i=0}^pv_i\partial_t^i\delta$ that belong to $V^\alpha$. The fact that for all $\alpha > 0$ we have
$$F_p{\mathcal M}(f^{-\alpha})=\tau_{\alpha}(F_{p+1}V^{\alpha})\quad\text{for all}\quad p\geq 0$$
is then precisely the content of Theorem~\ref{HI-Vfil}.
Since the Hodge filtration on $\Mmod(f^{-\alpha})$ is exhaustive, we deduce that $\tau_{\alpha}$ is surjective. 
\end{proof}

\begin{remark}
The same statement holds more generally when $D = {\rm div}(f)$ is not necessarily reduced, but $\alpha >0$ is such that $\lceil \alpha D \rceil$ is reduced. For this one simply needs to 
refer to \cite[Theorem~A]{MP3} instead.
\end{remark}

\subsection{Proof of the main result}
 We begin with the following general (and well-known) fact:

\begin{lemma}\label{nonpositive}
If $u\in\iota_+\shO_X$ is such that $\partial_tu\in V^{\alpha}$ for some $\alpha\leq 0$, then $u\in V^{\alpha+1}$.
\end{lemma}
\begin{proof}
Certainly if $\beta\ll 0$, then $u\in V^{\beta}$. We may assume that $u\neq 0$ and choose $\beta$ which is largest with this property, so that $u\not\in V^{>\beta}$. If $\beta\geq\alpha+1$, then we are done. Otherwise 
$\beta-1<\alpha\leq 0$, and $\partial_tu$ vanishes in ${\rm Gr}^{\beta-1}_V$. Recall however that an easy 
consequence of the definition of the $V$-filtration is that for every $\gamma \neq 0$, the map
$${\rm Gr}^{\gamma+1}_V\overset{\partial_t\cdot}\longrightarrow {\rm Gr}^{\gamma}_V$$
is bijective. It follows that $u$ vanishes in ${\rm Gr}_V^{\beta}$, a contradiction.
\end{proof}

Next, using the result of the previous section, we show that in order to bound the 
generation level of $\Mmod (f^{-\alpha})$ for any $\alpha >0$, it suffices to study the Hodge filtration on the associated graded terms ${\rm Gr}^{\beta}_V$, for special rational $\beta$.

\begin{corollary}\label{cor1}
If $\alpha\in (0,1]$ is a rational number and $q\geq 0$ is such that the Hodge filtration on ${\rm Gr}_V^{\beta}(\iota_+\shO_X)$ is generated at level $q+1$ for all $\beta\in 
[\alpha,1]$, then the Hodge filtration on ${\mathcal M}(f^{-\alpha})$ is generated at level $q$.
\end{corollary}
\begin{proof}
We need to show that $F_p{\mathcal M}(f^{-\alpha})\subseteq F_1\Dmod_X\cdot F_{p-1}{\mathcal M}(f^{-\alpha})$
for every $p>q$. Given such $p$ and $u\in F_p {\mathcal M}(f^{-\alpha})$,
it follows from Proposition~\ref{filtration_on_image} that we can find 
$\widetilde{u}\in F_{p+1}V^{\alpha}$ such that
$\tau_{\alpha}(\widetilde{u})=u$. The $V$-filtration is discrete,  hence after using the hypothesis finitely many times, we obtain
$$F_{p+1}V^{\alpha}\subseteq F_1\Dmod_X\cdot F_{p}V^{\alpha}+F_{p+1}V^{>1}.$$
Since $\tau_{\alpha}$ maps $F_1\Dmod_X\cdot F_{p}V^{\alpha}$ to $F_1\Dmod_X\cdot F_{p-1}{\mathcal M}(f^{-\alpha})$, we may clearly assume that 
$\widetilde{u}\in F_{p+1}V^{>1}$. In this case we can write $\widetilde{u}=tv$ for some $v\in F_{p+1}V^{>0}$; see for instance (the proof of) \cite[Lemma~4.5]{MP3}. Furthermore, by the definition of $F_{p+1} \iota_+\shO_X$, 
we can write $v=v_0\delta+\partial_tw$, for some $v_0 \in \shO_X$ and $w\in F_{p}\iota_+\shO_X$. 
Note that $\delta\in V^{>0}$, hence $v_0\delta\in V^{>0}$, and thus $\partial_tw\in V^{>0}$. By Lemma~\ref{nonpositive}, we have
$w\in F_{p}V^1$, so in particular $w\in F_{p}V^{\alpha}$. Since $tv_0\delta=v_0f\delta$, we have
$$u=\tau_{\alpha}(\widetilde{u})=\tau_{\alpha}(tv_0\delta+t\partial_tw)=(v_0f)f^{-\alpha}+\tau_{\alpha}(t\partial_tw)=(v_0f)f^{-\alpha}
+\alpha\cdot\tau_{\alpha}(w),$$
where the last equality follows from (\ref{eq1_lem1}) and (\ref{eq2_lem1}). But $(v_0f)f^{-\alpha}\in F_0{\mathcal M}(f^{-\alpha})$, which follows for example from Proposition~\ref{filtration_on_image}, since $f\delta\in V^{>1}\subseteq V^{\alpha}$ by (\ref{multiplier}).
Also, since $w\in F_{p}V^{\alpha}$, it follows from Proposition~\ref{filtration_on_image} that
$\tau_{\alpha}(w)\in F_{p-1}{\mathcal M}(f^{-\alpha})$.
We conclude that $u\in F_{p-1}{\mathcal M}(f^{-\alpha})$, completing the proof.
\end{proof}

We are finally able to give the proof of the main result:

\begin{proof}[Proof of Theorem~\ref{main3}]
According to Corollary \ref{cor1}, it suffices to know that ${\rm Gr}_V^{\beta}(\iota_+\shO_X)$ is generated at level $n-\lceil\widetilde{\alpha}_f+\alpha\rceil+1$ for all $\beta\in [\alpha,1]$. But this follows from Propositions ~\ref{crit1} and \ref{crit2}, which show that each ${\rm Gr}_V^{\beta}(\iota_+\shO_X)$ is generated at level $n-\lceil\widetilde{\alpha}_f+\beta\rceil+1$.
\end{proof}

\subsection{Proof of Theorem \ref{local_strong}}
Consider a reduced complex scheme $D$,
which can be embedded as a hypersurface in a smooth variety $X$, with minimal exponent 
$\widetilde{\alpha}_D$. 
We consider a resolution of singularities $\mu \colon \widetilde{D} \to D$. 
(Recall that by this we mean the disjoint union of resolutions of the irreducible components of $D$.) 
We further assume that $f$ is an isomorphism over the smooth locus of $D$ and that the reduced inverse image
of the singular locus $D_{\rm sing}$ of $D$ is a simple normal crossing divisor $E$ on $\widetilde{D}$.

We start with the following observation:

\begin{lemma}\label{independence}
The statement of Theorem \ref{local_strong} is independent of the choice of such a resolution.
\end{lemma}

\begin{proof}
A standard argument shows that it is enough to compare the assertion for $\mu$ and for another resolution with the same properties
of the form $\mu\circ g$, for some morphism $g\colon D'\to \widetilde{D}$. Note that if $E'$ is the reduced inverse image of $D_{\rm sing}$
on $D'$, then $E'=(g^*E)_{\rm red}$ and $g$ is an isomorphism over $\widetilde{D}\smallsetminus {\rm Supp}(E)$. In this case, we have
for all $i$
$$g_*\Omega_{D'}^i(\log E')=\Omega_{\widetilde{D}}^i(\log E)\quad\text{and}\quad R^q\Omega_{D'}^i(\log E')=0\quad\text{for all}\quad q>0$$
by \cite[Lemmas 1.2 and 1.5]{EV}; cf. also \cite[Theorem~31.1(i)]{MP1}. The assertion in the lemma thus follows via the Leray spectral sequence.
\end{proof}

If $D$ is smooth, then $\mu$ is an isomorphism, and we trivially have
$R^i \mu_* \Omega^j_Y (\log E) = 0$ for all $i > 0$ and all $j$. From now on, we focus on the case when $D$ is singular
(in which case recall, as mentioned in the Introduction, that $\widetilde{\alpha}_D \le n /2$, where $\dim(D)=n-1$). 

The proof of Theorem~\ref{local_strong}
 is inspired by the proof of \cite[Theorem~E]{MOP}, which partly treats the case $k=1$.
We begin with an auxiliary result:

\begin{lemma}\label{rel_van_blowup}
Let $g\colon Y\to X$ be the blow-up of a smooth variety $X$ along a smooth, irreducible subvariety $Z$,
of codimension $r\geq 2$. Let $F$ be a reduced simple normal crossing divisor on $X$, having
simple normal crossings with $Z$ as well, and denote by $\widetilde{F}$ the strict transform of
$F$  and by $E$ the exceptional divisor on $Y$. Then for every $i<r$, the following hold:
$$g_*\Omega^i_Y\big(\log (E+\widetilde{F})\big)=\Omega_X^i(\log F)\quad\text{and}\quad R^qg_*\Omega_Y^i\big(\log (E+F)\big)=0
\quad\text{for all}\quad q\geq 1.$$
\end{lemma}

\begin{proof}
For $i=0$ the assertion is clear and for
$i=1$ it follows from \cite[Theorem~31.1(ii)]{MP1}, so from now on we assume $i\geq 2$, hence $r\geq 3$.
We argue by induction on $r$. If $Z\subseteq {\rm Supp}(F)$, then the assertion holds for all $i$, using again
\cite[Lemmas 1.2 and 1.5]{EV}. Suppose now that $Z$ is not contained in ${\rm Supp}(F)$.
Since the assertion is local on $X$, we may assume that we have algebraic coordinates $x_1,\ldots,x_n$ on $X$
such that $Z$ is defined by $x_1,\ldots,x_r$ and all components of $F$ are defined by some $x_k$, with $k>r$. 
Let $T$ be the smooth divisor on $X$ defined by $x_1$ and consider the induced morphism $h\colon \widetilde{T}\to T$,
where $\widetilde{T}$ is the strict transform of $T$ on $Y$.  Consider the standard residue short exact sequence on $Y$:
\begin{equation}\label{exact_sequence}
0\to\Omega^i_Y\big(\log (E+\widetilde{F})\big)\to \Omega^i_Y\big(\log (E+\widetilde{F}+\widetilde{T})\big)
\to\Omega^{i-1}_{\widetilde{T}}\big(\log (E\vert_{\widetilde{T}}+\widetilde{F}\vert_{\widetilde{T}})\big)\to 0.
\end{equation}
Note that $h$ is the blow-up of $T$ along $Z$, with exceptional divisor $E\vert_{\widetilde{T}}$.
Moreover, the strict transform of $F\vert_T$ is $\widetilde{F}\vert_{\widetilde{T}}$. Since ${\rm codim}_T(Z)=r-1\geq 2$,
the inductive assumption thus gives
$$h_*\Omega^{i-1}_{\widetilde{T}}\big(\log (E\vert_{\widetilde{T}}+\widetilde{F}\vert_{\widetilde{T}})\big)=
\Omega_T^{i-1}(\log F\vert_T)\quad\text{and}$$
$$R^qh_*\Omega^{i-1}_{\widetilde{T}}\big(\log (E\vert_{\widetilde{T}}+\widetilde{F}\vert_{\widetilde{T}})\big)=0\quad\text{for all}\quad q\geq 1.$$
On the other hand, since $Z\subseteq {\rm Supp}(F+T)$ it follows, again from the reference above, that
$$g_*\Omega^i_Y\big(\log (E+\widetilde{F}+\widetilde{T})\big)=\Omega^i_X\big(\log (F+T)\big)\quad\text{and}$$
$$R^qg_*\Omega^i_Y\big(\log (E+\widetilde{F}+\widetilde{T})\big)=0\quad\text{for all}\quad q\geq 1.$$
The long exact sequence for higher direct images associated to (\ref{exact_sequence}) gives 
$$R^qg_*\Omega^i_Y\big(\log (E+\widetilde{F})\big)=0\quad\text{for all}\quad q\geq 2,$$
together with an exact sequence
$$0\to g_*\Omega^i_Y\big(\log (E+\widetilde{F})\big)\to \Omega^i_X\big(\log (F+T)\big)
\to\Omega_T^{i-1}(\log F\vert_T)$$
$$\to R^1g_*\Omega^i_Y\big(\log (E+\widetilde{F})\big)\to 0,$$
which compared to the standard residue sequence gives the assertions in the lemma.
\end{proof}
 
In order to apply the previous lemma, we will need to control the codimension of the blow-up centers when we have a lower bound on $\widetilde{\alpha}_D$. This is provided by:

\begin{proposition}\label{prop_sing_locus}
If $D$ is a singular effective divisor on $X$ such that $\widetilde{\alpha}_D>k$ for some nonnegative integer $k$, then
we have the following lower bound for the codimension of the singular locus $D_{\rm sing}$ of $D$:
$${\rm codim}_X(D_{\rm sing})\geq 2k+1.$$
\end{proposition}

To see this, we first prove a general lemma concerning the behavior of $\widetilde{\alpha}_D$ under restriction
to a general hypersurface.

\begin{lemma}\label{lem_gen_intersection}
If $D$ is an effective divisor on $X$ and $H$ is a general smooth hypersurface in $X$ (for example, a general 
member of a basepoint-free linear system), then 
$$\widetilde{\alpha}_{D\vert_H} \geq \widetilde{\alpha}_D.$$
\end{lemma}

\begin{proof}
We may assume that $D$ is reduced: otherwise ${\rm lct}(X,D)<1$, hence ${\rm lct}(X,D)=\widetilde{\alpha}_{D}$ and 
for $H$ general we have
$$\widetilde{\alpha}_{D\vert_H}\geq {\rm lct}(H,D\vert_H)\geq{\rm lct}(X,D),$$
where the second inequality follows, for example, from the Generic Restriction theorem for multiplier ideals,
see \cite[Theorem~9.5.35]{Lazarsfeld}. Supposing now that $D$ is reduced, we appeal to results on Hodge ideals (for $\QQ$-divisors). If we write $\widetilde{\alpha}_D=p+\alpha$,
for some $\alpha\in (0,1]$ and some nonnegative integer $p$,  it follows from \cite[Corollary~C]{MP3} 
that $I_p(\alpha D)=\shO_X$ and since $H$ is general, according to \cite[Theorem~13.1]{MP2} we have
$$I_p(\alpha D\vert_H)=I_p(\alpha D)\cdot\shO_H=\shO_H.$$
Another application of \cite[Corollary~C]{MP3} gives $\widetilde{\alpha}_{D\vert_H}\geq p+\alpha=\widetilde{\alpha}_D$.
\end{proof}

\begin{proof}[Proof of Proposition~\ref{prop_sing_locus}]
We may assume that $X$ is an affine variety. We denote $r=\dim(D_{\rm sing})$. If $r\geq 1$ and $H$ is a general 
hyperplane section of $X$, then $H$ is smooth, $D\vert_H$ is singular, and $\dim\big((D\vert_H)_{\rm sing}\big)=r-1$. Moreover, it follows from
Lemma~\ref{lem_gen_intersection} that $\widetilde{\alpha}_{D\vert_H}>k$. After iterating this $r$ times, we obtain a smooth
subvariety $Y$ of $X$, with $\dim(Y)=n-r$, such that $D\vert_Y$ is a singular effective divisor and 
$\widetilde{\alpha}_{D\vert_Y}>k$.
Since $\widetilde{\alpha}_{D\vert_Y}\leq\frac{1}{2}\dim(Y)$, we conclude that $k<\frac{1}{2}(n-r)$, hence
$${\rm codim}_X(D_{\rm sing})=n-r\geq 2k+1.$$
\end{proof}

We can finally approach our main goal for this section.

\begin{proof}[Proof of Theorem~\ref{local_strong}]
Let $X$ be a smooth variety in which $D$ embeds as a hypersurface. We need to show, equivalently, that if $k$ is a nonnegative integer such that
$\widetilde{\alpha}_D>k$, then
$$R^{n-1-i}  \mu_* \Omega^i_{\widetilde{D}} (\log E) = 0 \,\,\,\,\,\,{\rm for~all} \,\,\,\,\,\,i\leq k.$$

By Lemma~\ref{independence}, the assertion in the theorem is independent of the choice of resolution $\mu$. 
We thus first construct a log resolution $\mu\colon Y\to X$ of the pair $(X,D)$, as a composition
$$Y=X_N\overset{\mu_N}\longrightarrow X_{N-1}\longrightarrow\cdots\longrightarrow X_1\overset{\mu_1}\longrightarrow X_0=X,$$
where 
\begin{enumerate}
\item[i)] Each $\mu_j$ with $1\leq j\leq N$ is the blow-up of a smooth, irreducible subvariety $Z_{j-1}$ of $X_{j-1}$ that lies over $D_{\rm sing}\subseteq X$.
We denote by $F_j$ the exceptional divisor of $X_j\to X$ and by $D_j$ the strict transform of $D$ on $X_j$.
\item[ii)] Each $Z_{j-1}$ with $1\leq j\leq N$ has simple normal crossings with $D_{j-1}+F_{j-1}$.
\end{enumerate}
In particular, we see inductively
that each $X_j$ is smooth and $F_j+D_j$ is a simple normal crossing divisor. 
We may assume that $\widetilde{D}=D_N$ is smooth, so that the induced morphism $\varphi\colon \widetilde{D}\to D$ is a resolution of $D$
that is an isomorphism over $D\smallsetminus D_{\rm sing}$. 
Furthermore, if
$F=F_N$, and $E=F\vert_{\widetilde{D}}$, then $E=\mu^{-1}(D_{\rm sing})_{\rm red}$ and this is a simple normal crossing divisor on $\widetilde{D}$.

\noindent {\bf Claim}. For every $i\leq 2k$, we have
\begin{equation}\label{eq2_version_local_vanishing}
\mu_*\Omega_Y^i(\log F)=\Omega_X^i \,\,\,\,\,\,{\rm and} \,\,\,\,\,\,R^q\mu_*\Omega^i_Y(\log F)=0\,\,\,\,{\rm for ~all}\,\,\,\,q\geq 1.
\end{equation}
To see this, using the Leray spectral sequence, it is enough to show that for every $1\leq j\leq N$ we have
\begin{equation}\label{eq3_version_local_vanishing}
{\mu_j}_*\Omega_{X_j}^i(\log F_j)=\Omega_{X_{j-1}}^i(\log F_{j-1})\,\,\,\,{\rm and}\,\,\,\, R^q{\mu_j}_*\Omega_{X_j}^i(\log F_j)=0\,\,\,\,{\rm for ~all}\,\,\,\, q\geq 1.
\end{equation}
If $Z_{j-1}\subseteq F_{j-1}$, then this follows from \cite[Lemmas 1.2 and 1.5]{EV} (or \cite[Theorem~31.1(i)]{MP1}).
On the other hand, if $Z_{j-1}\not\subseteq F_{j-1}$, then 
$Z_{j-1}$ is equal to the strict transform of its image in $X$.
By construction and Proposition~\ref{prop_sing_locus}, it follows that ${\rm codim}_{X_{j-1}}(Z_{j-1})\geq 2k+1$,
and (\ref{eq3_version_local_vanishing}) then follows from Lemma~\ref{rel_van_blowup}. This proves our claim.

Consider now the residue short exact sequence
$$0\to\Omega_Y^{i+1}(\log F)\to\Omega_Y^{i+1}\big(\log (\widetilde{D}+F)\big)\to\Omega_{\widetilde{D}}^{i}(\log E)\to 0$$
on $Y$, and the following piece in the corresponding long exact sequence for higher direct images:
$$R^{n-1-i}\mu_*\Omega_Y^{i+1}\big(\log (\widetilde{D}+F)\big)\to R^{n-1-i}\varphi_*\Omega_{\widetilde{D}}^i(\log E)
\to R^{n-i}\mu_*\Omega_Y^{i+1}(\log F).$$
Since 
$$i \le k < \widetilde{\alpha}_D \le n /2,$$
the first term vanishes because of Corollary~\ref{main2}.
Since the third term vanishes by the above Claim, we conclude that the middle term vanishes as well. This completes the proof of the theorem.
\end{proof}


\section*{References}
\begin{biblist}


\bib{BS}{article}{
 author = {Budur, Nero},
 author={Saito, Morihiko},
 title = {Multiplier ideals, {$V$}-filtration, and spectrum},
  journal = {J. Algebraic Geom.},
  volume = {14},
      date= {2005},
    number= {2},
     pages = {269--282},
}

		\bib{EV}{article}{
   author={Esnault, H{\'e}l{\`e}ne},
   author={Viehweg, Eckart},
   title={Rev\^etements cycliques},
   conference={
      title={Algebraic threefolds},
      address={Varenna},
      date={1981},
   },
   book={
      series={Lecture Notes in Math.},
      volume={947},
      publisher={Springer, Berlin-New York},
   },
   date={1982},
   pages={241--250},
      }


\bib{KS}{article}{
      author={Kebekus, Stefan},
      author={Schnell, Christian},
      title={Extending holomorphic forms from the regular locus of a complex space to a resolution of singularities},
      journal={preprint arXiv:1811.03644}, 
      date={2018}, 
}

\bib{Lazarsfeld}{book}{
       author={Lazarsfeld, Robert},
       title={Positivity in algebraic geometry II},  
       series={Ergebnisse der Mathematik und ihrer Grenzgebiete},  
       volume={49},
       publisher={Springer-Verlag, Berlin},
       date={2004},
}      

\bib{MOP}{article}{
      author={Musta\c t\u a, Mircea},
      author={Olano, Sebasti\'an},
      author={Popa, Mihnea},
      title={Local vanishing and Hodge filtration for rational singularities},
      journal={preprint 	arXiv:1703.06704, to appear in J. Inst. Math. Jussieu}, 
      date={2017}, 
}


\bib{MP1}{article}{
      author={Musta\c t\u a, Mircea},
      author={Popa, Mihnea},
      title={Hodge ideals},
      journal={preprint arXiv:1605.08088, to appear in Memoirs of the AMS}, 
      date={2016}, 
}


\bib{MP2}{article}{
      author={Musta\c t\u a, Mircea},
      author={Popa, M.},
      title={Hodge ideals for $\QQ$-divisors: birational approach},
      journal={J. de l'\'Ecole Polytechnique}, 
      volume={6},
      date={2019},
      pages={283--328},
}


\bib{MP3}{article}{
      author={Musta\c t\u a, Mircea},
      author={Popa, Mihnea},
      title={Hodge ideals for $\QQ$-divisors, $V$-filtration, and minimal exponent},
      journal={preprint arXiv:1807.01935}, 
      date={2018}, 
}

\bib{Popa}{article}{
      author={Popa, Mihnea},
      title={$\Dmod$-modules in birational geometry},
      journal={Proceedings of the ICM 2018, Rio de Janeiro, vol.2, World Scientific}, 
      date={2018},
      pages={781--806}, 
}


\bib{Saito-MHP}{article}{
   author={Saito, M.},
   title={Modules de Hodge polarisables},
   journal={Publ. Res. Inst. Math. Sci.},
   volume={24},
   date={1988},
   number={6},
   pages={849--995},
}

\bib{Saito_duality}{article}{
   author={Saito, Morihiko},
   title={Duality for vanishing cycle functors},
   journal={Publ. Res. Inst. Math. Sci.},
   volume={25},
   date={1989},
   number={6},
   pages={889--921},
}


\bib{Saito-MHM}{article}{
   author={Saito, M.},
   title={Mixed Hodge modules},
   journal={Publ. Res. Inst. Math. Sci.},
   volume={26},
   date={1990},
   number={2},
   pages={221--333},
}

\bib{Saito-B}{article}{
   author={Saito, Morihiko},
   title={On $b$-function, spectrum and rational singularity},
   journal={Math. Ann.},
   volume={295},
   date={1993},
   number={1},
   pages={51--74},
}


\bib{Saito_microlocal}{article}{
   author={Saito, M.},
   title={On microlocal $b$-function},
   journal={Bull. Soc. Math. France},
   volume={122},
   date={1994},
   number={2},
   pages={163--184},
}

      \bib{Saito-HF}{article}{
   author={Saito, Morihiko},
   title={On the Hodge filtration of Hodge modules},
   journal={Mosc. Math. J.},
   volume={9},
   date={2009},
   number={1},
   pages={161--191},
}

\bib{Saito-MLCT}{article}{
      author={Saito, M.},
	title={Hodge ideals and microlocal $V$-filtration},
	journal={preprint arXiv:1612.08667}, 
	date={2016}, 
}


\bib{Zhang}{article}{
      author={Zhang, Mingyi},
      title={Hodge filtration and Hodge ideals for $\QQ$-divisors with weighted homogeneous isolated singularities},
      journal={preprint arXiv:1810.06656}, 
      date={2018}, 
}



\end{biblist}
\end{document}